\RequirePackage[l2tabu, orthodox]{nag}

\documentclass[12pt]{amsart}
\usepackage{fullpage,url,amssymb,enumerate,colonequals}
\usepackage{mathrsfs} 

\usepackage[OT2,T1]{fontenc}
\DeclareSymbolFont{cyrletters}{OT2}{wncyr}{m}{n}
\DeclareMathSymbol{\Sha}{\mathalpha}{cyrletters}{"58}

\usepackage{color}

\newcommand{\defi}[1]{\textsf{#1}} 

\newcommand{\G}{\mathbb{G}}

\newcommand{\PP}{\mathbb{P}}
\newcommand{\Q}{\mathbb{Q}}

\newcommand{\Z}{\mathbb{Z}}
\newcommand{\Qbar}{{\overline{\Q}}}

\newcommand{\EE}{\mathscr{E}}

\newcommand{\LL}{\mathscr{L}}
\newcommand{\OO}{\mathscr{O}}

\DeclareMathOperator{\Char}{char}

\DeclareMathOperator{\Disc}{Disc}

\DeclareMathOperator{\Div}{Div}

\DeclareMathOperator{\Pic}{Pic}

\DeclareMathOperator{\Res}{Res}

\DeclareMathOperator{\Span}{Span}
\DeclareMathOperator{\Spec}{Spec}

\newcommand{\red}{{\operatorname{red}}}

\newcommand{\Directsum}{\bigoplus} 

\newcommand{\isom}{\simeq}

\newcommand{\tensor}{\otimes} 

\newcommand{\isomto}{\overset{\sim}{\rightarrow}}

\newtheorem{theorem}{Theorem}[section]
\newtheorem{lemma}[theorem]{Lemma}
\newtheorem{corollary}[theorem]{Corollary}
\newtheorem{proposition}[theorem]{Proposition}

\theoremstyle{definition}

\newtheorem{question}[theorem]{Question}
\newtheorem{conjecture}[theorem]{Conjecture}
\newtheorem{example}[theorem]{Example}

\theoremstyle{remark}
\newtheorem{remark}[theorem]{Remark}

\makeatletter
\g@addto@macro\bfseries{\boldmath} 
\makeatother

\usepackage{microtype}

\usepackage[
	pagebackref,
	pdfauthor={Bradley W. Brock, Bruce W. Jordan, Bjorn Poonen, Anthony J. Scholl, Joseph L. Wetherell}, 
	pdftitle={Linear independence in linear systems on elliptic curves},
]{hyperref}
\usepackage[alphabetic,backrefs,lite,nobysame]{amsrefs} 

\begin{document}

\title{Linear independence in linear systems \\ on elliptic curves}

\author{Bradley W. Brock}
\address{Center for Communications Research, 805 Bunn Drive, Princeton, 
NJ 08540-1966, USA}
\email{bwbrock@idaccr.org}

\author{Bruce W. Jordan}
\address{Department of Mathematics, Baruch College, The City University
of New York, One Bernard Baruch Way, New York, NY 10010-5526, USA}
\email{bruce.jordan@baruch.cuny.edu}

\author{Bjorn Poonen}
\address{Department of Mathematics, Massachusetts Institute of Technology, Cambridge, MA 02139-4307, USA}
\email{poonen@math.mit.edu}
\urladdr{\url{http://math.mit.edu/~poonen/}}

\author{Anthony J. Scholl}
\address{Department of Pure Mathematics and Mathematical Statistics, Centre for Mathematical Sciences, Wilberforce Road, Cambridge CB3 0WB, UK}
\email{a.j.scholl@dpmms.cam.ac.uk}

\author{Joseph L. Wetherell}
\address{Center for Communications Research, 4320 Westerra Court, San Diego, CA 92121-1969, USA}

\thanks{B.P. was supported in part by National Science Foundation grant DMS-1601946 and Simons Foundation grants \#402472 (to Bjorn Poonen) and \#550033.}

\date{May 11, 2020}

\begin{abstract}
Let $E$ be an elliptic curve, with identity $O$, 
and let $C$ be a cyclic subgroup of odd order $N$,
over an algebraically closed field $k$ with $\Char k \nmid N$.
For $P \in C$, let $s_P$ be a rational function 
with divisor $N \cdot P - N \cdot O$.
We ask whether the $N$ functions $s_P$ are linearly independent.
For generic $(E,C)$, we prove that the answer is yes.
We bound the number of exceptional $(E,C)$ when $N$ is a prime
by using the geometry of the universal generalized elliptic curve over $X_1(N)$.
The problem can be recast in terms of sections of 
an arbitrary degree $N$ line bundle on $E$.
\end{abstract}

\maketitle

\section{Introduction}\label{S:introduction}

Fix $N \ge 1$ 
and an algebraically closed field $k$ such that $\Char k \nmid N$.
Let $E$ be an elliptic curve over $k$.
Let $C \subset E$ be a cyclic subgroup of order $N$.

Let $\LL$ be a degree $N$ line bundle on $E$.
Since $\Pic^0(E)$ is divisible,
there exist points $P \in E$ such that $\OO(N \cdot P) \isom \LL$,
or equivalently, such that there exists a global section $s_P$ of $\LL$
whose divisor of zeros is $N \cdot P$.
The set of such $P$ is a coset $E[N]'$ of $E[N]$.
Let $C' \subset E[N]'$ be a coset of $C$.
Then $\# C' = N$.
On the other hand, $\dim \Gamma(E,\LL) = N$ by the Riemann-Roch theorem.

\begin{question}
\label{Q:linearly independent}
Are the sections $s_P$ for $P \in C'$ linearly independent in $\Gamma(E,\LL)$?
\end{question}

The answer is sometimes yes, sometimes no.

\begin{example}
Let $O \in E(k)$ be the identity.
Let $\LL = \OO(N \cdot O)$ and $C' = C$.
Then $s_P$ is a rational function on $E$ 
with divisor $(s_P) = N \cdot P - N \cdot O$.
Question~\ref{Q:linearly independent} asks 
whether the $s_P$ for $P \in C$ are linearly independent,
i.e., whether they form a basis of $\Gamma(E,\OO(N \cdot O))$.
\end{example}

\begin{proposition}
\label{P:codimension does not depend on L and C'}
The answer to Question~\ref{Q:linearly independent}
depends only on $(E,C)$, not on the choice of degree $N$ line bundle $\LL$
or coset $C'$ or $s_P$ for $P \in C'$.
More precisely, 
the codimension of $\Span\{s_P : P \in C'\}$ in $\Gamma(E,\LL)$ 
depends only on $(E,C)$.
\end{proposition}

We will prove Proposition~\ref{P:codimension does not depend on L and C'}
in Section~\ref{S:codimension}.

The pair $(E,C)$ corresponds to a $k$-point on 
the classical modular curve $Y_0(N)$.

\begin{theorem}
\label{T:finitely many}
Let $N$ be an odd positive integer such that $\Char k \nmid N$.
Then for all but finitely many $(E,C) \in Y_0(N)(k)$,
Question~\ref{Q:linearly independent} has a positive answer.
\end{theorem}

We next work towards a quantitative version of Theorem~\ref{T:finitely many},
at least for prime $N$.
Let $c_{(E,C)}$ be the codimension 
in Proposition~\ref{P:codimension does not depend on L and C'}, 
and let $D = \sum_{(E,C)} c_{E,C} \, (E,C) \in \Div Y_0(N)$.

\begin{theorem}
\label{T:D1 and D2}
Let $N > 3$ be a prime with $\Char k \nmid N$.
There exist effective divisors $D_1$ and $D_2$ on $Y_0(N)$ 
such that $D=D_1+2D_2$ with 
\begin{align*}
  \deg D_1 &\le (N^2-1)/24\\
  \deg D_2 &\le (N-3)(N^2-1)/48.
\end{align*}
\end{theorem}

\begin{conjecture}
\label{C:char 0}
If $\Char k=0$, then $D_1$ and $D_2$ are reduced and disjoint,
and the equalities in Theorem~\ref{T:D1 and D2} are equalities.
\end{conjecture}

\begin{remark}
Conjecture~\ref{C:char 0} is equivalent to the claim
that for prime $N > 3$ and $\Char k = 0$, 
there are exactly $(N^2-1)/24$ points $(E,C) \in Y_0(N)(k)$ with $c_{E,C}=1$, 
exactly $(N-3)(N^2-1)/48$ points with $c_{E,C}=2$, 
and no points with $c_{E,C} > 2$.
\end{remark}

The primes $N > 3$ for which the genus of $X_0(N)$ is $0$
are $5$, $7$, and $13$;
for these we checked that Conjecture~\ref{C:char 0} is true,
using methods to be described in Section~\ref{S:examples}.
There we will also show that Conjecture~\ref{C:char 0} sometimes fails 
when $\Char k > 0$.

\section{Notation}
\label{S:notation}

Let $\mu$ be the group of roots of unity in $k$.
Fix a primitive $N$th root of unity $\zeta \in k$.

If $C$ is a finite abelian group with $\Char k \nmid \#C$,
and $V$ is a $k$-representation of $C$,
and $\chi \colon C \to k^\times$ is a character,
define the $\chi$-isotypic subspace
\[
	V^\chi \colonequals \{v \in V: cv = \chi(c) \, v \textup{ for all $c \in C$}\}.
\]

Let $X$ be a regular $k$-variety.
Let $\Div X$ be its divisor group.
Now suppose in addition that $X$ is integral.
Let $k(X)$ be its function field.
If $f \in k(X)^\times$, let $(f)=(f)_X \in \Div X$ be its divisor.
For each irreducible divisor $Z$ on $X$, let $v_Z$ be the associated valuation.
A finite morphism of regular integral curves $\phi \colon X \to Y$
induces a homomorphism $\phi_* \colon \Div X \to \Div Y$ 
(sending each point to its image)
compatible with the norm homomorphism 
$\phi_* \colon k(X)^\times \to k(Y)^\times$.

\section{Codimension is independent of choices}
\label{S:codimension}

\begin{proof}[Proof of Proposition~\ref{P:codimension does not depend on L and C'}]
Fix $(E,C)$.
Once $\LL$ and $C'$ are also fixed, each $s_P$ is determined 
up to scaling by an element of $k^\times$,
which does not change the span.

For each $Q \in E(k)$, let $\tau_Q \colon E \to E$
be the morphism sending $x$ to $x+Q$.
Pulling back by $\tau_Q$ shows that the codimension for $(\LL,C')$
is the same as for $(\tau_Q^* \LL, \tau_Q^{-1}(C'))$.
If $Q \in E[N]$, then $\tau_Q^* \LL \isom \LL$
but $\tau_Q^{-1}(C')$ can be any other coset of $C'$ in $E[N]'$;
thus the codimension is independent of $C'$.
As $Q$ ranges over $E(k)$, the line bundle $\tau_Q^* \LL$
ranges over all degree $N$ line bundles;
thus the codimension is independent of $\LL$ too.
\end{proof}

\section{Normalized functions}
\label{S:normalized}

If $f \in k(E)^\times$ has divisor supported on $E[N]$,
then $[N]_* (f) = 0$, so $[N]_* f \in k^\times$.
Multiplying $f$ by a constant $a \in k^\times$ multiplies $[N]_* f$ by $a^{\deg\,[N]}=a^{N^2}$.
Call $f \in k(E)^\times$ \defi{normalized}
if there exists $N \ge 1$ such that $[N]_* f \in \mu$.
In that case, $[N']_* f \in \mu$ for all multiples $N'$ of $N$.
Therefore the normalized functions form a subgroup of $k(E)^\times$.
Given a principal divisor supported on torsion points,
there exists a normalized function with that divisor,
uniquely determined up to multiplication by a root of unity.
In particular, a normalized constant rational function is an element of $\mu$.
If $f$ is normalized and $P$ is a torsion point on $E$,
then $\tau_P^* f$ is normalized too.

\section{Character-weighted combinations}
\label{S:characters}

{}From now on, we assume that $N$ is odd.  
View $C$ as a degree $N$ divisor on $E$.
Choose $\LL \colonequals \OO(C)$.
The group $C$ acts on $\LL$: each $P$ acts as $\tau_P^*$ on sections of $\LL$.
Since $N$ is odd, $\LL \isom \OO(N \cdot O)$.
Choose $C'=C$.
Choose sections $s_P$ as in Section~\ref{S:introduction}.

If we view $s_O$ as a rational function on $E$, then $(s_O) = N \cdot O - C$.
Assume that $s_O$ is normalized.
For $P \in C'=C$, we may assume that $s_P \colonequals \tau_{-P}^* s_O$.
Then $\Span\{s_P: P \in C\}$ 
is the image of a $kC$-module homomorphism $kC \to \Gamma(E,\LL)$,
so it decomposes as a direct sum of distinct characters.
For each character $\chi \colon C \to k^\times$,
the projection of $\Span\{s_P: P \in C\}$ onto $\Gamma(E,\LL)^\chi$
is spanned by
\[
	g_\chi 
	\colonequals \left( \sum_{P \in C} \chi(P) \tau_{-P}^* \right) s_O 
	= \sum_{P \in C} \chi(P) \, s_P.
\]
Then $c_{E,C} = \# \{\chi : g_\chi = 0 \}$.

\begin{lemma}
\label{L:s_O symmetry}
We have $[-1]^* s_O = s_O$.
\end{lemma}

\begin{proof}
The divisor $(s_O)$ is fixed by $[-1]^*$,
so $s_O$ is an eigenvector of $[-1]^*$, with eigenvalue $\pm 1$.
Since $v_O(s_O)$ is even, the eigenvalue is $1$.
\end{proof}

\begin{lemma}
\label{L:g symmetry}
For each $\chi$, we have $[-1]^* g_\chi = g_{\chi^{-1}}$.
\end{lemma}

\begin{proof}
Apply
\[
	[-1]^* \left( \sum_{P \in C} \chi(P) \tau_{-P}^* \right)
	= \left( \sum_{P \in C} \chi(P) \tau_{P}^* \right) [-1]^*
	= \left( \sum_{Q \in C} \chi(-Q) \tau_{-Q}^* \right) [-1]^*
\]
to $s_O$ and use Lemma~\ref{L:s_O symmetry}.
\end{proof}

\begin{lemma}
\label{L:product of s_P}
We have $\prod_{P \in C} s_P \in \mu$.
\end{lemma}

\begin{proof}
It is a normalized rational function whose divisor is $0$.
\end{proof}

\section{An almost canonical basis}
\label{S:canonical basis}

Fix $(E,C)$.
Let $\phi \colon E \to E'$ be an isogeny with kernel $C$.
Let $\hat{\phi} \colon E' \to E$ be the dual isogeny.
The Weil pairing
\[
	e_\phi \colon \ker \phi \times \ker \hat{\phi} \to k^\times
\]
is nondegenerate, so choosing $Q \in \ker \hat{\phi}$ is equivalent to 
choosing a character $\chi \colon C \to k^\times$,
related via $\chi(P) = e_\phi(P,Q)$ for all $P \in C$.
Let $C_\chi = \phi^* Q \in \Div E$.
Let $h_\chi$ be a normalized function with $(h_\chi) = C_\chi - C$.

\begin{lemma}
\label{L:Weil pairing}
For $P \in C$, we have $\tau_P^* h_\chi = \chi(P) \, h_\chi$.
\end{lemma}

\begin{proof}
This is the definition of $e_\phi(P,Q)$, which equals $\chi(P)$; 
see \cite{SilvermanAEC2009}*{Exercise~3.15(a)}.
\end{proof}

Thus $0 \ne h_\chi \in \Gamma(E,\LL)^\chi$ for all $\chi$,
but $\Directsum_\chi \Gamma(E,\LL)^\chi$ is $N$-dimensional, 
so $\Gamma(E,\LL)^\chi = kh_\chi$.
In particular, $g_\chi/h_\chi \in k$.
Now
\begin{equation}
\label{E:counting}
	c_{E,C} = \# \{\chi : g_\chi = 0 \} = \# \{\chi: g_\chi/h_\chi = 0\}.
\end{equation}

\begin{lemma}
\label{L:h symmetry}
For each $\chi$, we have $[-1]^* h_\chi \equiv h_{\chi^{-1}} \pmod{\mu}$.
\end{lemma}

\begin{proof}
Compare divisors, and observe that both sides are normalized.
\end{proof}

\begin{lemma}
\label{L:gh symmetry}
For any $\chi$, 
we have $g_\chi/h_\chi \equiv g_{\chi^{-1}}/h_{\chi^{-1}} \pmod{\mu}$.
\end{lemma}

\begin{proof}
By Lemmas \ref{L:g symmetry} and~\ref{L:h symmetry},
$[-1]^*(g_\chi/h_\chi) \equiv g_{\chi^{-1}}/h_{\chi^{-1}} \pmod{\mu}$.
On the other hand, $g_\chi/h_\chi$ is constant on $E$, 
so $[-1]^*(g_\chi/h_\chi) = g_\chi/h_\chi$.
\end{proof}

\section{The universal elliptic curve}
\label{S:universal}

Given an elliptic curve $E$ over $k$
and a point $P \in E(k)$ of exact order $N$,
we define $C$ as the subgroup generated by $P$.
For $m \in \Z/N\Z$, let $\chi \colon C \to k^\times$
be the character such that $\chi(P)=\zeta^m$,
and set $g_m \colonequals g_\chi$ and $h_m \colonequals h_\chi$.
We may assume that $h_0=1$.

Suppose that $N > 3$ and $\Char k \nmid N$.
Then the moduli space $Y_1(N)$ parametrizing pairs $(E,P)$
is a fine moduli space 
(it can be viewed as an \'etale quotient of the affine curve $Y(N)$
constructed by Igusa \cite{Igusa1959},
because a pair $(E,P)$ consisting of an elliptic curve 
and a point of exact order $N > 3$ has no nontrivial automorphisms).
Thus there is a universal elliptic curve $\EE \to Y_1(N)$.
The construction of $s_O$ makes sense on $\EE$,
except that normalizing it may require taking an $N^2$th root
of an invertible function on $Y_1(N)$.
Thus $s_O$ is a rational function not on the elliptic surface $\EE \to Y_1(N)$,
but on a pullback $\EE' \to Y_1(N)'$ 
by some finite \'etale cover $Y_1(N)' \to Y_1(N)$.
Then $s_O^n$ for some $n \ge 1$ lies in $k(\EE)^\times$, 
and $s_O$ itself may be identified with
$\frac{1}{n} \tensor s_O^n \in \Q \tensor_{\Z} k(\EE)^\times$.
Its divisor $(s_O)$ is then an element of $\Q \tensor \Div \EE$.
Given $m \in \Z/NZ$,
we may also define $g_m, h_m \in k(\EE')^\times$
and consider them as elements of $\Q \tensor k(\EE)^\times$.
Then $g_m/h_m$ is a \emph{regular} function on $Y_1(N)'$
and we may consider it an as element of $\Q \tensor k(Y_1(N))^\times$.
Its divisor on $Y_1(N)$ lies in $\Div Y_1(N)$, 
not just $\Q \tensor \Div Y_1(N)$,
since $Y_1(N)' \to Y_1(N)$ is finite \'etale.

\section{The universal generalized elliptic curve}
\label{S:universal generalized}

We continue to assume $N > 3$.
Complete $Y_1(N)$ to a smooth projective curve $X_1(N)$ over $k$.
One can recover from \cite{Deligne-Rapoport1973}*{IV.4.14 and VI.2.7}
that $\EE \to Y_1(N)$ 
can be completed to a ``universal generalized elliptic curve'' 
$\pi \colon \overline{\EE} \to X_1(N)$.
The following description of the cusps of $X_1(N)$ 
and the associated Tate curves is well-known;
see \cite{Deligne-Rapoport1973}*{VII.2}
and \cite{Faltings-Jordan1995}*{\S3.1}.

The cusps on $X_1(N)$ are in bijection with 
\[
	\coprod_{d|N} \frac{(\Z/d\Z)^\times \times (\Z/e\Z)^\times}{\{\pm 1\}},
\]
where $e = N/d$ in each term.
The integer $e$ equals the ramification index of $X_1(N) \to X(1)$ at the cusp,
and is called the width of the cusp.
The cusp represented by $(d,a,b)$, where $0 \le a < d$ and $0 \le b < e$
and $\gcd(a,d)=\gcd(b,e)=1$, 
has a uniformizer $q$
and a punctured formal neighborhood $\Spec k((q))$
above which is the Tate curve analytically 
isomorphic to $\left( \G_m/q^{e\Z},\zeta^a q^b \right) \in Y_1(N)(k((q)))$.
This Tate curve specializes above the cusp itself
to an $e$-gon consisting of irreducible components 
$Z_i \isom \PP^1$ indexed by $i \in \Z/e\Z$
such that $0 \in Z_i$ is attached to $\infty \in Z_{i+1}$ for all $i$.
We choose the coordinate $t \colon Z_i \isomto \PP^1$ for each $i$
such that a point $t_i q^i + \sum_{j>i} t_j q^j \in \G_m/q^{e\Z}$ 
with $t_i \in k^\times$ specializes to 
$t_i \in \G_m \subseteq \PP^1 \isom Z_i \subset \pi^{-1}(y)$.
For each cusp $y$, 
define $F_y \colonequals \pi^*y = \sum_i Z_i \in \Div \overline{\EE}$.

\section{Divisors}

Given a rational function $f$ on $\EE$
whose divisor on $\EE$ is known,
the divisor of $f$ on $\overline{\EE}$ is determined 
up to addition of a linear combination of the $F_y$.
We now explain how to compute it, modulo the ambiguity.
Fix a cusp $y$ of $X_1(N)$, and let $q$ be a uniformizer at $y$,
and let $Z_0,\ldots,Z_{e-1}$ be the components of $\pi^{-1}(y)$.
The valuations $n_i \colonequals v_{Z_i}(f)$ can be simultaneously computed, 
modulo addition of a constant independent of $i$,
by the relations $(f/q^{n_i}).Z_i = 0$ for all $i$,
which amount to linear equations in the $n_i$.
Let us make these equations explicit.
In the case where the zeros and poles of $f$ specialize
to smooth points of $\pi^{-1}(y)$, let $r_i$ be the number
of them specializing to a point of $Z_i$,
counted with multiplicity, with poles counted as negative.
In the equation $(f/q^{n_i}).Z_i = 0$,
only $Z_{i+1}$, $Z_{i-1}$, and the horizontal divisors in $(f)$
meet $Z_i$, so the equation says 
\[
	(n_{i+1}-n_i) + (n_{i-1} - n_i) + r_i = 0.
\]
There is one such equation for each $i$.
Solving this system of $e$ equations 
yields all the $n_i$ up to a common additive constant,
since the solutions to the corresponding homogeneous system
are the arithmetic progressions that are periodic modulo $N$,
i.e., constant sequences.
If in addition, $f$ is normalized, then $\sum n_i = 0$;
now the $n_i$ are uniquely determined.

The above procedure can be applied also to any $f \in \Q \tensor k(\EE)^\times$,
and in particular to the functions $s_P$, $g_m$, and $h_m$.

\begin{lemma}
\label{L:s_O}
For $f=s_O$, 
\begin{enumerate}[\upshape (a)]
\item At a cusp of $X_1(N)$ above $\infty \in X_0(N)$,
we have $e=1$, $n_0=0$, 
and $s_O|_{Z_0} = (1-t)^N/(1-t^N)$ in $\Q \tensor k(Z_0)^\times$.
\item At a cusp of $X_1(N)$ above $0 \in X_0(N)$,
we have $e=N$, $n_i = (N^2-1)/12 - i(N-i)/2$ for $0 \le i < N$,
and $\left( q^{(N^2-1)/24} s_O \right)|_{Z_{(N-1)/2}}$ 
has a zero at $\infty$ and not at $0$,
while $\left( q^{(N^2-1)/24} s_O\right)|_{Z_{(N+1)/2}}$ 
has a zero at $0$ and not at $\infty$.
\end{enumerate}
\end{lemma}

\begin{proof}
\hfill
\begin{enumerate}[\upshape (a)]
\item 
A cusp above $\infty$ has a punctured neighborhood 
above which is the Tate curve $\G_m/q^\Z$ with cyclic subgroup $\mu_N$,
specializing to a $1$-gon.
In fact, the relation $\prod_{R \in C} \tau_R^* s_O = 1$ 
in $\Q \tensor k(\EE)^\times$
from Lemma~\ref{L:product of s_P}
implies $N n_0 = 0$, so $n_0 = 0$.

The order $N$ zero of $s_O$ specializes to $1$,
and the $N$ poles of $s_O$ specialize to the $N$th roots of unity,
so $s_O|_{Z_0}$ is a nonzero scalar times $(1-t)^N/(1-t^N)$.

Since $s_O$ is normalized, $[N]_* s_0 \in \mu$.
On the other hand, the morphism $[N]$ specializes to the
$N$th power map on $Z_0 \isom \PP^1$,
which pushes $(1-t)^N/(1-t^N)$ forward to 
the norm $\prod_{\omega \in \mu_N} (1-\omega t)^N/(1-(\omega t)^N)
= (1-t^N)^N/(1-t^N)^N = 1$.
By the previous two sentences, 
the scalar of the previous paragraph is in $\mu$.
\item 
A cusp above $0$ has a punctured neighborhood above which is the
Tate curve $\G_m/q^{N\Z}$ with cyclic subgroup generated by $q$.
The $N$ zeros specialize to $Z_0$, 
but the $N$ poles specialize to different $Z_i$, one pole per $Z_i$.
Thus $r_0=N-1$ and $r_i=-1$ for $i \ne 0$.
On the other hand, $\prod_{R \in C} \tau_R^* s_O = 1$ implies $\sum n_i = 0$.
Together these imply that $n_i = (N^2-1)/12 - i(N-i)/2$ for $0 \le i < N$.
The most negative of these are $n_{(N-1)/2}$ and $n_{(N+1)/2}$,
which are both $-(N^2-1)/24$.

The divisor of $\left( q^{(N^2-1)/24} s_O \right)|_{Z_{(N-1)/2}}$ 
on $Z_{(N-1)/2} \isom \PP^1$ is 
\[
	(n_{(N+1)/2} - n_{(N-1)/2})(0) + (n_{(N-3)/2} - n_{(N-1)/2})(\infty) - (1) = (\infty) - (1).
\]
Similarly, the divisor of $\left( q^{(N^2-1)/24} s_O \right)|_{Z_{(N+1)/2}}$ 
on $Z_{(N+1)/2}$ is
\[
	(n_{(N+3)/2} - n_{(N+1)/2})(0) + (n_{(N-1)/2} - n_{(N+1)/2})(\infty) - (1) = (0) - (1).\qedhere
\]
\end{enumerate}
\end{proof}

\begin{corollary}
\label{C:g_m}
\hfill
\begin{enumerate}[\upshape (a)]
\item At the cusp above $\infty \in X_0(N)$
given by $(\G_m/q^\Z,\zeta)$,
we have $g_0|_{Z_0} = N$, 
and for $m \ne 0$ 
we have $g_m|_{Z_0} = (-1)^m N \binom{N}{m} t^m / (1-t^N)$, 
in $\Q \tensor k(\Z_0)^\times$.
\item 
At a cusp above $0$, for any $m,i \in \Z/N\Z$, 
we have $v_{Z_i}(g_m) = -(N^2-1)/24$.
\end{enumerate}
\end{corollary}

\begin{proof}
\hfill
\begin{enumerate}[\upshape (a)]
\item 
Up to a root of unity which may be ignored, 
$s_O|_{Z_0} = (1-t)^N/(1-t^N)$ by Lemma~\ref{L:s_O}(a).
Translation by $P$ restricts to multiplication by $\zeta$ on $Z_0$, so 
\begin{align*}
	s_{jP}|_{Z_0}
	&= \tau_{-jP}^* s_O|_{Z_0} \\
	&= (1-\zeta^{-j} t)^N/(1-(\zeta^{-j} t)^N) \\
	&= \frac{1}{1-t^N} \sum_{i=0}^N \binom{N}{i} (-1)^i \zeta^{-ij} t^i \\
	g_m|_{Z_0}
	&= \sum_{j=0}^{N-1} \zeta^{mj} \frac{1}{1-t^N} \sum_{i=0}^N \binom{N}{i} (-1)^i \zeta^{-ij} t^i \\
	&= \frac{1}{1-t^N} \sum_{i=0}^N (-1)^i \binom{N}{i} t^i \sum_{j=0}^{N-1} \zeta^{(m-i)j} \\
	&= \frac{1}{1-t^N} \sum_{i=0}^N (-1)^i \binom{N}{i} t^i 
	\begin{cases}
          N, & \textup{ if $m-i \equiv 0 \pmod{N}$;} \\
	  0, & \textup{ otherwise.}
        \end{cases}
\end{align*}
If $m=0$, then only the terms with $i=0$ or $i=N$ are nonzero,
and the sum becomes $(1-t^N)N$.
If $m \ne 0$, then only the term with $i=m$ is nonzero,
and the sum becomes $(-1)^m \binom{N}{m} t^m N$.
\item 
The translation action of $C$ acts simply transitively on the
set of components $Z_i$ above the cusp.
Thus the numbers $v_{Z_i}(s_{jP})$ for $j=0,\ldots,N-1$
equal the numbers $v_{Z_{i'}}(s_O)$ for $i'=0,\ldots,N-1$ in some order,
which are described by Lemma~\ref{L:s_O}(b).
Hence in the sum $g_m = \sum_{j=0}^{N-1} \zeta^{mj} s_{jP}$
there are exactly two terms with the most negative valuation along $Z_i$, 
so $v_{Z_i}(\zeta^{mj} s_{jP}) = -(N^2-1)/24$ for $j=j_1$ and $j=j_2$, say.
The last two claims in Lemma~\ref{L:s_O}(b)
imply that one of the functions 
$(q^{(N^2-1)/24} \zeta^{mj} s_{jP})|_{Z_i}$ for $j=j_1$ and $j=j_2$
has a zero at $\infty$ and not at $0$,
while the other has a zero and not at $\infty$,
so their sum is nonzero on $Z_i$.
Thus $v_{Z_i}(g_m) = -(N^2-1)/24$ too.\qedhere
\end{enumerate}
\end{proof}

\begin{proof}[Proof of Theorem~\ref{T:finitely many}]
We may work on the finite cover $Y_1(N)'$ of $Y_0(N)$ 
defined in Section~\ref{S:universal}.
By Corollary~\ref{C:g_m}(b), no $g_m$ is identically zero.
Hence each function $g_m/h_m$ on $Y_1(N)'$ has only finitely many zeros.
Equation~\eqref{E:counting} shows that outside the union of these zeros, 
$c_{E,C}=0$; i.e., the $f_P$ are linearly independent.
\end{proof}

Let $G \colonequals g_1 g_2 \cdots g_{N-1}$ 
and $H \colonequals h_1 h_2 \cdots h_{N-1}$
in $\Q \tensor k(\EE)^\times$.
The divisor of $H$ on $\EE$ is $\EE[N] - N C$.

\begin{lemma}
\label{L:H}
For $f=H$,
\begin{enumerate}[\upshape (a)]
\item At a cusp of $X_1(N)$ above $\infty \in X_0(N)$, 
we have $e=1$ and $n_0=-(N^2-1)/12$.
\item At a cusp of $X_1(N)$ above $0 \in X_0(N)$,
we have $n_i=0$ for all $i$.
\end{enumerate}
\end{lemma}

\begin{proof}
We work on the universal generalized elliptic curve over $X(N)$,
whose degenerate fibers are all $N$-gons,
so that the zeros and poles of $H$ do not specialize to 
the singular points of fibers.
As usual, let $Z_0,\ldots,Z_{N-1}$ be the components above a cusp;
let $n'_i = v_{Z_i}(H)$.
The normalization implies that the product of all translates of $H$
by $N$-torsion points is in $\mu$,
so $\sum n_i=0$.
\begin{enumerate}[\upshape (a)]
\item 
We have $r_0=-N(N-1)$ and $r_i=N$ for $i\ne 0$.
The $r_i$ here are $-N$ times the $r_i$ in the proof of Lemma~\ref{L:s_O}(b),
so the resulting $n'_i$ are also multiplied by $-N$;
that is, $n'_i = -N(N^2-1)/12 + Ni(N-i)/2$ for $0 \le i < N$.
Finally, $X(N) \to X_1(N)$ has ramification index $N$ 
at cusps above $\infty$, so $n_0 = n_0'/N$.
\item 
Each $h_m$ has one zero and one pole specializing to each $Z_i$,
so $r_i=0$ for all $i$.
Thus $n_i'=0$ for all $i$, so $n_i=0$ for all $i$.\qedhere
\end{enumerate}
\end{proof}

\begin{lemma}
\label{L:ratios}
Let $N > 3$ be prime.
\hfill
\begin{enumerate}[\upshape (a)]
\item 
The element $g_0 = g_0/h_0 \in \Q \tensor k(\EE)^\times$
lies in $\Q \tensor k(X_0(N))^\times$,
its valuations at the cusps of $X_0(N)$ are
$v_\infty(g_0) = 0$ and $v_0(g_0)=-(N^2-1)/24$, 
and its divisor on $Y_0(N)$ is effective and of degree $(N^2-1)/24$.
\item 
The $G/H = \prod_{m=1}^{N-1}(g_m/h_m) \in \Q \tensor k(\EE)^\times$
lies in $\Q \tensor k(X_0(N))^\times$,
with $v_\infty(G/H) \ge (N^2-1)/12$ and $v_0(G/H)=-(N-1)(N^2-1)/24$.
The divisor of $G/H$ on $Y_0(N)$ is of degree $\le (N-3)(N^2-1)/24$,
and it is twice an effective divisor on $Y_0(N)$.
\end{enumerate}
\end{lemma}

\begin{proof}
Each $g_m/h_m$ is constant on each elliptic curve fiber,
so $g_m/h_m$ lies in $\Q \tensor k(X_1(N))^\times$.
The Galois group of $X_1(N) \to X_0(N)$ fixes $g_0/h_0$
and permutes the $g_m/h_m$,
so $g_0/h_0$ and $G/H$ are in $\Q \tensor k(X_0(N))^\times$.

\begin{enumerate}[\upshape (a)]
\item 
The valuations $v_\infty(g_0)$ and $v_0(g_0)$ 
are determined by Corollary~\ref{C:g_m}.
On the other hand, (a power of) $g_0=g_0/h_0$ is regular on $Y_0(N)$,
and its divisor on the projective curve $X_0(N)$ has degree $0$.
\item 
The valuation of $G/H$ along the component $Z_0$ above a cusp of $X_1(N)$
above $\infty$
is $\ge \left( \sum_{m=1}^{N-1} 0 \right) - (-(N^2-1)/12) = (N^2-1)/12$,
by Corollary~\ref{C:g_m}(a) and Lemma~\ref{L:H}(a);
thus $v_\infty(G/H) \ge (N^2-1)/12$.
The valuation of $G/H$ along any component $Z_i$ above a cusp above $0$
is $\left( \sum_{m=1}^{N-1} -(N^2-1)/24 \right) - 0 = -(N-1)(N^2-1)/24$
by Corollary~\ref{C:g_m}(b) and Lemma~\ref{L:H}(b);
thus $v_0(G/H) = -(N-1)(N^2-1)/24$.

Since the divisor of $G/H$ on $X_0(N)$ has degree $0$,
its divisor on $Y_0(N)$ has degree at most
$-(N^2-1)/12 + (N-1)(N^2-1)/24 = (N-3)(N^2-1)/24$.

That it is twice an effective divisor 
can be checked on the \'etale cover $Y_1(N)'$ of Section~\ref{S:universal}.
There, each $g_m/h_m$ is regular, and Lemma~\ref{L:gh symmetry}
shows that $g_{-m}/h_{-m} = g_m/h_m$,
so $G/H$ is a square.
\qedhere
\end{enumerate}
\end{proof}

\begin{proof}[Proof of Theorem~\ref{T:D1 and D2}]
Let $D_{Y_1(N)}$ be the pullback of $D$ under $Y_1(N) \to Y_0(N)$.
Let $(g_m/h_m)_{\red} \in \Div Y_1(N)$ be the reduced divisor 
whose support equals the divisor of $g_m/h_m$ on $Y_1(N)$.
Equation~\eqref{E:counting} says that 
$D_{Y_1(N)} = \sum_{m=0}^{N-1} (g_m/h_m)_{\red}$.
The divisors
$D_{Y_1(N),1} \colonequals (g_0/h_0)_{\red}$
and $D_{Y_1(N),2} = \sum_{m=1}^{(N-1)/2} (g_m/h_m)_{\red} 
= \frac{1}{2} \sum_{m=1}^{N-1} (g_m/h_m)_{\red}$
are invariant under the Galois group of $Y_1(N) \to Y_0(N)$,
so they are pullbacks of divisors $D_1$ and $D_2$ on $Y_0(N)$.
We have $D_{Y_1(N)} = D_{Y_1(N),1} + 2 D_{Y_1(N),2}$, so $D = D_1 + 2 D_2$.

The degree of $D_1$ is bounded by the degree of $g_0/h_0$ on $Y_0(N)$,
which is $(N^2-1)/24$ by Lemma~\ref{L:ratios}(a).
Similarly, the degree of $2D_2$ is bounded by the degree of 
$G/H$ on $Y_0(N)$, which is at most $(N-3)(N^2-1)/24$
by Lemma~\ref{L:ratios}(b).
\end{proof}

\section{Examples}
\label{S:examples}

Let $N > 3$ be prime.
On the Tate curve over $k((q))$ analytically isomorphic to $\G_m/q^\Z$ 
we can write down a function with prescribed divisor
in terms of theta functions in $u$ and $q$, 
where $u$ is the coordinate on $\G_m$.
In this way, we express the elements $s_P$, $g_m$, and $h_m$
in terms of $u$ and $q$
and we compute the $q$-expansions of 
the rational functions $g_0/h_0$ and $G/H$ on $X_0(N)$.

Now suppose in addition that the genus of $X_0(N)$ is $0$;
that is, $N \in \{5,7,13\}$.
Let $\eta(q) = q^{1/24} \prod_{n \ge 1} (1-q^n)$.
Then the function $(N^{1/2} \eta(q^N)/\eta(q))^{24/(N-1)}$
is the $q$-expansion of a rational function $t$ on $X_0(N)$ 
with $k(t) = k(X_0(N))$
such that $t$ has a zero at the cusp $\infty$ and a pole at the cusp $0$.
Because of Lemma~\ref{L:ratios}, 
this lets us compute $g_0/h_0$ and $G/H$ as polynomials 
$f_1(t)$ and $t^{(N^2-1)/12} f_2(t)$
whose zeros with $t \ne 0$ give the points $(E,C) \in Y_0(N)$ 
with $c_{E,C}>0$; call these points \defi{exceptional}.
Moreover, in these cases, using an expression for $j$ in terms of $t$,
we may take the $k(t)/k(j)$ norm and take numerators 
to obtain polynomials $F_1(j)$ and $F_2(j)$ (determined up to scalar multiple)
whose zeros are the $j$-invariants of the $E$ 
such that $c_{E,C}>0$ for some $C \subset E$.

For $N \in \{5,7,13\}$, 
we found that the polynomials $f_1(t)$ and $f_2(t)$
are of degrees $(N^2-1)/24$ and $(N-3)(N^2-1)/48$
and have disjoint distinct roots in $\Qbar$ 
(in fact, they are irreducible over $\Q$);
this verifies Conjecture~\ref{C:char 0} for these values of $N$.
In fact, $F_1(j)$ and $F_2(j)$ had the same properties.

\begin{example}
Let $N=5$.
Then
\begin{align*}
	f_1(t) &= t+5 \\
	f_2(t) &= t+10 \\
	F_1(j) &= j-1600 \\
	F_2(j) &= 2j+25.
\end{align*}
Each of $f_1$ and $f_2$ has a unique zero, 
and these zeros are distinct,
and they avoid the cusps (where $t=0$ and $t=\infty$),
except in characteristic~$2$ (we always exclude characteristic~$5$).
Thus in characteristics $\ne 2,5$,
we have $c_{E,C}=0$ 
except for one $(E,C)$ with $c_{E,C}=1$ and one $(E,C)$ with $c_{E,C}=2$,
so the conclusion of Conjecture~\ref{C:char 0} for $N=5$ 
holds in characteristics $\ne 2,5$.
In characteristic~$2$,
we have $c_{E,C}=0$ 
except for one $(E,C)$ with $c_{E,C}=1$,
so the conclusion of Conjecture~\ref{C:char 0} fails.

Moreover, in characteristics $\ne 2,5$, 
the two exceptional $(E,C)$ have $j$-invariants $1600$ and $-25/2$,
which are distinct except in characteristics $3$ and $43$.
In characteristics $3$ and $43$, we find that $c_{E,C}=0$ always 
except that the $E$ with $j(E)=1600=-25/2$
has two exceptional subgroups $C_1$ and $C_2$,
with $c_{E,C_1}=1$ and $c_{E,C_2}=2$.
\end{example}

\begin{example}
Let $N=7$.
Then
\begin{align*}
	f_1(t) &= t^2 + 7 t + 7 \\
	f_2(t) &= t^4 + 21 t^3 + 168 t^2 + 588 t + 735 \\
	F_1(j) &= j^2 - 1104 j - 288000 \\
	F_2(j) &= 15 j^4 - 28857 j^3 + 20163177 j^2 - 5403404499 j - 141176604743
\intertext{and the constant terms, discriminants, and resultants factor as follows:}
	f_1(0) &= 7 \\
        f_2(0) &= 3 \cdot 5 \cdot 7^2 \\
	\Disc(f_1) &= 3 \cdot 7 \\
	\Disc(f_2) &= - 3^3 \cdot 7^6 \\
	\Res(f_1,f_2) &= 7^4 \\
	\Disc(F_1) &= 2^8 \cdot 3^3 \cdot 7^3 \\
	\Disc(F_2) &= - 3 \cdot 7^{18} \cdot 43^2 \cdot 139^2 \cdot 421^2 \cdot 591751^2 \\
	\Res(F_1,F_2) &= 5 \cdot 7^{12} \cdot 47 \cdot 3491 \cdot 5939 \cdot 244603.
\end{align*}
The values of $f_1(0)$, $f_2(0)$, $\Disc(f_1)$, $\Disc(f_2)$
show that in all characteristics $\ne 3,5,7$,
we have $c_{E,C}=0$ 
except for two $(E,C)$ with $c_{E,C}=1$ and four with $c_{E,C}=2$,
so the conclusion of Conjecture~\ref{C:char 0} for $N=7$ 
holds in characteristics $\ne 3,5,7$.
In characteristic~$3$, we have $c_{E,C}=0$
except that $c_{E,C}=1$ for one $(E,C)$ 
(corresponding to the double root $t=1$ of $f_1$, where $j(E)=0$).
In characteristic~$5$, we have $c_{E,C}=0$
except for two $(E,C)$ with $c_{E,C}=1$
and only \emph{three} $(E,C)$ with $c_{E,C}=2$.

Moreover, excluding characteristic~$7$ as always, 
the exceptional $(E,C)$ have distinct values of $j(E)$ 
except in characteristics $2$, $43$, $47$, $139$, $421$, $3491$, 
$5939$, $244603$, and $591751$,
for which there are exactly two exceptional $(E,C)$ sharing the same $j(E)$.
In characteristic~$2$, these two have $c_{E,C}=1$ 
(since $2$ divides $\Disc(F_1)$ but not $\Disc(f_1)$)
In characteristics $43$, $139$, $421$, and $591751$,
these two have $c_{E,C}=2$
(since these primes divide $\Disc(F_2)$ but not $\Disc(f_2)$).
In characteristics $47$, $5939$, and $244603$,
these two have $c$-values $1$ and $2$, respectively
(since these primes divide $\Res(F_1,F_2)$ but not $\Res(f_1,f_2)$).
\end{example}

\begin{example}
Let $N=13$.
Then $\deg f_1 = \deg F_1 = 7$ and $\deg f_2 = \deg F_2 = 35$,
and each of the four polynomials has distinct zeros in $\Qbar$.
The analysis is similar to that for $N=5$ and $N=7$, 
except that we were unable to factor $\Disc(F_2)$ completely.
\end{example}

\section*{Acknowledgment} 

We thank Hyuk Jun Kweon for a comment on a draft of the article.

\end{document}